\documentclass[intlimits]{amsart}
\usepackage[all]{xy}
\usepackage{amssymb}

\newcommand{\y}{{\bf y}}
\newcommand{\x}{{\bf x}}

\newcommand{\disc}{\operatorname{disc}}

\newcommand{\R}{{\bf R}}

\newcommand{\Q}{{\bf Q}}

\newcommand{\ord}{\operatorname{ord}}

\newcommand{\Z}{\bold{Z}}
\newcommand{\N}{\bold{N}}

\newcommand{\fo}{{\mathfrak o}}

\newcommand{\gen}{\operatorname*{gen}}
\newcommand{\A}{\mathbb{A}}

\newcommand{\genus}{\text{gen}}
\newcommand{\spn}{\operatorname*{spn}}
\newcommand{\Spin}{\operatorname{Spin}}
\newtheorem{theorem}{Theorem}
\newtheorem{lemma}[theorem]{Lemma}
\newtheorem{definition}[theorem]{Definition}

\newtheorem{corollary}[theorem]{Corollary}
\newtheorem{proposition}[theorem]{Proposition}

\begin{document}
\title[Local conditions for global representations]{Local conditions for
global  representations of quadratic forms}  
\author[R. Schulze-Pillot]{Rainer Schulze-Pillot} 
\thanks{
  representation of quadratic forms, integral quadratic forms}
\maketitle
\begin{abstract}
{We show that the theorem of Ellenberg and Venkatesh on representation
  of integral quadratic forms by integral positive definite quadratic
  forms is valid under weaker conditions on the represented form.}
\end{abstract}

In the article \cite{ev} Ellenberg and Venkatesh prove that for any
integral positive definite quadratic form $f$ in $n$ variables there
is a constant $C(f)$ such that
integral quadratic forms $g$  of square free
discriminant in $m\le n-5$ variables with minimum
$\mu(g)>C(f)$ are represented by  $f$ if and only if they are represented
by $f$ locally everywhere, i.e., over $\R$ and over all $p$-adic integers. 
If one fixes an odd prime $p$ not
dividing the discriminant of $f$ one can find
a constant $C'(f,p)$ such that representability is even guaranteed for
$g$ of rank $m\le
n-3$ with  $\mu(g)>C'(f,p)$, provided the discriminant of $g$ is further
restricted to be prime to $p$. It is mentioned in \cite{ev} that I have
suggested to replace the condition of square free discriminant on $g$
by a weaker condition.  This suggestion is worked out here. Combining
our version of the result of \cite{ev} with results of Kitaoka we
also obtain some new cases in which with a suitable fixed prime $q$
the only condition on $g$ (apart 
from $\mu(g)>C(f,q)$ and representability of $g$ by $f$ locally everywhere)
is bounded divisibility of the discriminant
of $g$ by $q$. Moreover, results on extensions of representations as
given in \cite{boe-ra,ckko} can be obtained with new dimension bounds. 
We
take the occasion to reformulate some of the proofs of \cite{ev} in
a way that is closer to other work on the subject.

\vspace{3mm}
We will work throughout in the language of lattices as described
e.g.\ in \cite{Ki,OM} (with the exception of Theorem
\ref{matrixversion}). We fix a totally 
real number field $F$ with ring of integers ${\mathfrak o}$ and a
totally positive definite quadratic space  $(V,Q)$ over $F$ of
dimension $n\ge 3$; the quadratic form $Q$ may be written as
$Q(x)=\langle x,x\rangle$ with a 
scalar product $\langle \quad,\quad\rangle$ on $V$.  By $O_V(F)$ we
denote the group of isometries of $V$ with respect to $Q$ (the
orthogonal group of the quadratic space $(V,Q)$), by $O_V(\A)$ its
adelization, by $SO_V(F)$ resp.\ $SO_V(\A)$ their subgroups of
elements of determinant $1$. For a lattice $\Lambda$ on $V$ we denote
its automorphism group (or unit group) $\{\sigma \in O_V(F) \mid
\sigma(\Lambda)=\Lambda\}$ by $O_{\Lambda}(\fo)$, similarly for the
local or adelic analogues. The minimum of $\Lambda$ is
$$\mu(\Lambda):=\min\{N^F_{\Q}(Q({\bf x}))\mid {\bf x} \in \Lambda, {\bf
  x}\ne {\bf 0}\}.$$  For the
question which lattices have large minimum it does not
matter whether we chose this definition or $\min\{Tr^F_{\Q}(Q({\bf
  x}))\mid {\bf x} \in \Lambda, {\bf   x}\ne {\bf 0}\}$ instead, see
the remark in \cite[p.139]{hkk}.  

\begin{definition} 
Let $\tilde{F}$ denote either $F$ or one of its completions $F_w$ at
a finite place $w$ of $F$ and  $\tilde{{\mathfrak o}}$
either ${\mathfrak o}$ or its completion ${\mathfrak o_w}$ at the same
place $w$, analogously we write $\tilde{V}$ for $V$ or its completion
$V_w$ at $w$. 
Let $\Lambda$ be an $\tilde{{\mathfrak o}}$-lattice of full rank $n=\dim_{\tilde{F}}(\tilde{V})$
on $\tilde{V}$ and $M$ an   $\tilde{{\mathfrak o}}$-lattice of rank $m$ on
a quadratic space $(\tilde{W}',Q')$ of dimension $m$ over
$\tilde{F}$.

\vspace{2mm}
The lattice {\em $M$ is represented by $\Lambda$} if there is an 
$\tilde{F}$-linear embedding $\varphi:\tilde{W}' \longrightarrow
\tilde{V}$ with 
$Q(\varphi(w))=Q'(w)$ for all $w \in \tilde{W}'$ (the embedding is {\em
  isometric}) for which $\varphi(M)\subseteq \Lambda$.

\vspace{2mm}
If $M$ is represented by $\Lambda$ through the isometric embedding
$\varphi$ we put
$\tilde{W}:=\varphi(\tilde{W}')$ and  say that 
the representation $\varphi$ is {\em primitive} if one has
$\tilde{W}\cap \Lambda=\varphi(M)$.\\
The representation
$\varphi$  of $M$ by
$\Lambda$ has  {\em imprimitivity bounded 
by $c \in \tilde{{\mathfrak o}}$} if $c v \in \varphi(M)$ for all $v \in
\tilde{W} \cap \Lambda$.
In  particular for $c=1$ a representation with imprimitivity bounded
by $1$ is the same as a primitive representation.

\vspace{2mm} We say that the $\fo$-lattice $M$ on the $F$-space $W'$
is represented 
(resp.\ primitively represented resp.\ represented with imprimitivity
bounded by $c \in \fo$) by the $\fo$-lattice $\Lambda$ on the
$F$-space $V$ locally everywhere if the completion $M_w$ is
represented (resp.\ primitively represented resp.\ represented with
imprimitivity bounded by $c$) 
by the completion $\Lambda_w$ for all (finite or infinite) places $w$ of $F$. 
\end{definition}    
{\it Remark.} Two representations $\varphi:M\longrightarrow
\Lambda,\,\psi:M\longrightarrow \Lambda'$ by lattices $\Lambda,
\Lambda'$ on $V$ of the $\fo$-lattice $M$ are
in the same genus of representations if they are in the same
$O_V(\A_F)$-orbit, i.\ e., if for every place $w$ of $F$ on has
$\rho_w \in O_V(F_w)$ with $\rho_w(\Lambda_w)=\Lambda'_w$ 
and $\psi_w=\rho_w \varphi_w$ (where $\varphi_w, \psi_w$
are the extensions of $\varphi, \psi$ to isometries from $M_w=\fo_wM$ to
$\Lambda_w=\fo_w\Lambda$ resp.  $\Lambda'_w=\fo_w\Lambda'$). Obviously
two representations in 
the same representation genus have the same primitivity/imprimitivity
properties.

In our setting the given local representations determine a genus of
representations of $M$ by lattices in the genus of $\Lambda$, and all
representations of $M$ that we are going to construct will lie in that genus of
representations.

\vspace{3mm}
As in \cite{ev} we have an action of the adelic orthogonal group
$O_V(\A)$ on the set of lattices on $V$, the isometry classes of
lattices in the orbit of $\Lambda$ under
this action are (by the Hasse-Minkowski theorem) precisely the isometry
classes of lattices in the  genus $\genus(\Lambda)$ of
$\Lambda$. The isometry classes of lattices in the orbit of
the subgroup $O'_V(\A)$ of adelic transformations of determinant and
spinor norm $1$  (or equivalently under the adelic version
$\Spin_V(\A)$ of the spin group of $V$) are the classes in  the spinor genus (or
spin genus) $\spn(\Lambda)$ of $\Lambda$. As usual we write
$\theta(\varphi)$ for 
the spinor norm of a special orthogonal transformation $\varphi$.

\begin{lemma}\label{spinrep}
Let $\Lambda$ be an ${\mathfrak o}$-lattice of rank $n$ on $V$ and $M$
an $\fo$-lattice on $W'$ of rank $m\le n-3$ and assume that $M$ is
represented by $\Lambda$ locally everywhere with imprimitivity bounded
by the number $c \in \fo$.

Then $M$ is represented (globally) with imprimitivity bounded by $c$
by some lattice $\Lambda'$ in the spinor genus of $\Lambda$.
\end{lemma}
\begin{proof}
The same statement without the
condition on bounded imprimitivity is proved in
\cite{kneser_mz,hsia-pacific}, the proof here is essentially the same:

The Hasse-Minkowski theorem allows to assume $M\subseteq W'\subseteq
V$ and hence the set $T$ of places $w$ of $F$ for which $M_w
\not\subseteq \Lambda_w$ is finite and consists only of nonarchimedean
places. 
By adding  the finitely many places $w$ to $T$ for
which $M_w$ is not primitive in $\Lambda_w$ 
we can even assume that $M_w$ is a
primitive sublattice of $\Lambda_w$ for all $w
\not\in T$.

By assumption for the $w\in T$ we have isometric embeddings
$\varphi_w:W'\longrightarrow V$ such that $\varphi_w(M_w)\subseteq
\Lambda_w$ and $c \cdot \Lambda_w \cap \varphi_w(W'_w)\subseteq
\varphi_w(M_w)$ hold, and by Witt's theorem we can assume the
$\varphi_w$ to  be elements of the orthogonal group $O_V(F_w)$; it is
easy to see that we can assume them to have determinant $1$.  Since
the orthogonal complement $U_w$ of $\varphi_w(W'_w)$ in $V_w$ has
dimension $\ge 3$ its special orthogonal group contains
transformations of arbitrary spinor norm, see \cite{OM}. Choosing a
transformation $\psi_w\in SO_U(F_w)$ with
$\theta(\psi_w)=\theta(\varphi_w)$ and extending $\psi_w$ to $V_w$ by
letting it act as identity on $\varphi_w(W'_w)$ we obtain
$\rho_w=\psi_w\circ \varphi_w \in SO_V(F_w)$ with $\theta(\rho_w)=1$
and such that  $\rho_w(M_w)\subseteq
\Lambda_w$ and $c \cdot \Lambda_w \cap \varphi_w(W'_w)\subseteq
\rho_w(M_w)$ hold. The lattice $\Lambda'$ on $V$ defined by
$\Lambda'_w=\Lambda_w$ for $w \not\in T$ and
$\Lambda'_w=\rho_w^{-1}\Lambda_w$ for $w \in T$ is then in the spinor
genus of $\Lambda$ and contains $M$ with imprimitivity bounded by $c$.
\end{proof}

{\it Remark.} 
\begin{enumerate}
\item The argument allows to include congruence conditions on the
  representation of $M$ by $\Lambda'$: If $S$ is a finite set of
  non-archimedean places of $F$ and representations
  $\sigma_v:M_v\longrightarrow \Lambda_v$ are fixed for the $v\in S$,
  one can choose for any $s \in \N$ the representation
  $\rho:M\longrightarrow \Lambda'$ in such a way that 
$$\rho(m)-\sigma_v(m)\in p^s\Lambda'_v \text{ for all } v \in S$$
holds for all $m \in M$. This is proved as above, adding the places of
$S$ to the finite set $T$ of places chosen above.
\item If the quadratic space $(V,Q)$ is not totally definite
but indefinite at at least one archimedean place of $F$, the strong
approximation theorem for the spin group \cite{eichler_qfog} implies
that the spinor genera of lattices on $V$ coincide with the isometry
classes. In that case the question of global representations is
therefore completely settled by the Lemma, see
\cite{eichler_spin,kneser_mz,hsia-pacific}, 
and the only case needing further
attention is the totally definite case.
\end{enumerate}
For definite $(V,Q)$ we can use (following Eichler and Kneser) the
strong approximation theorem with respect to some finite place $w$ of
$F$ for  which the completion $V_w=V\otimes F_w$ is isotropic and
extend the lattices on $V$ to  ``arithmetically indefinite'' lattices
over the larger ring $\fo^{(w)}=F\cap F_w\prod_{v \ne
  w}\fo_w$; over this ring spinor genera and isometry classes coincide
again, and we obtain (similarly as in \cite{hkk}):
\begin{lemma}\label{p-rep}
With notations and assumptions as in Lemma \ref{spinrep} let $w$ be a
finite place of $F$ for which the completion $V_w=V\otimes F_w$ is
isotropic. 

Then every isometry class in the spinor genus of $\Lambda$ has a
representative $\tilde{\Lambda}\subseteq V$ with
$\tilde{\Lambda}_v=\Lambda_v$ for all finite places $v \ne w$
of $F$ and 
$\tilde{\Lambda}_w\in \Spin_V(F_w)\Lambda_w$.

Furthermore there is a lattice $\Lambda''$ in the spinor genus of $\Lambda$ with
$\Lambda''_v=\Lambda_v$ for all finite places $v \ne w$ of $F$ and
$\Lambda''_w\in \Spin_V(F_w)\Lambda_w$ such
that $M$ is represented by $\Lambda''$ with imprimitivity bounded by $c$.
\end{lemma}
\begin{proof} To prove the second assertion 
we take first a lattice $\Lambda'$ in the spinor genus of $\Lambda$ as
guaranteed by Lemma \ref{spinrep} and may assume that $M$ is a
sublattice of $\Lambda'$ with imprimitivity bounded by $c$.  We write
$\Lambda'=\varphi \Lambda$ 
with an adele $\varphi=(\varphi_v) \in \Spin_V(\A_F)$. The strong
approximation theorem allows to write $\varphi=\sigma \psi$ with
$\sigma \in \Spin_V(F)$ and  $\psi=(\psi_v) \in \Spin_V(\A_F)$ with
$\psi_v\Lambda_v=\Lambda_v$ for all finite $v \ne w$.
Setting $\Lambda''=\sigma^{-1}\Lambda'=\psi\Lambda$ we have
$\Lambda''_v=\Lambda_v$ for all finite places $v \ne w$ of $F$, and
$\Lambda''$ contains the lattice $\sigma^{-1}(M)$, which is isometric
to $M$, with imprimitivity bounded by $c$. The same argument shows the
first assertion of the Lemma. 
\end{proof}
{\it Remark.} Again we may include congruence conditions at finitely
many non-archi\-me\-dean places $v\ne w$.
\vspace{3mm}
\begin{lemma}\label{orbits}
Let $w$ be a non-archimedean place of $F$ and $\Lambda_w$ be an
$\fo_w$-lattice of rank $n$ on $V_w=V\otimes F_w$. Let ${\mathcal W}$
be a set  of regular subspaces of $V_w$, put  $M_W:=W\cap \Lambda_w$ for $W
\in{\mathcal W}$ and assume that the (additive) $w$-adic valuation
$\ord_w(disc(M_W))$ of the discriminants of the $M_W$ is bounded by
some $j \in \N$ independent of $W$.

Then the set ${\mathcal W}$ is contained in the union of  finitely
many orbits under 
the action of the compact open subgroup 
$$\tilde{K}_w := \Spin_{\Lambda_w}(\fo_w)=\{ \tau \in \Spin_V(F_w)\mid 
\tau(\Lambda_w)=\Lambda_w\}$$ of $\Spin_V(F_w)$.
\end{lemma}
\begin{proof}
This is stated and proved (with $\Spin_{\Lambda_w}(\fo_w)$ replaced by
some unspecified compact subset of $\Spin_V(F_w)$) as Lemma 2 in \cite{ev}. We
show here that it is  an immediate consequence of known facts from the
local theory of quadratic  forms:

The boundedness condition on  $\ord_w(\disc(M_W))$ leaves only finitely
many possibilities for the Jordan decomposition of $M_W$ and hence for
the isometry class of $M_W$.
It is well known (see Satz
30.2 of \cite{kneser_buch}, the proof given there for lattices over
$\Z$ applies for $\fo_w$-lattices as well) that for a given $M$ there are only
finitely many $O_{\Lambda_w}(\fo_w)$-orbits of representations of $M$ by
$\Lambda_w$, in other words, only finitely many orbits of sublattices of
$\Lambda_w$ which are isometric to $M$. Obviously, each of these orbits
breaks up into finitely many orbits under the action of
$\Spin_{\Lambda_w}(\fo_w)$. 
\end{proof}
\begin{proposition}\label{claim}
Let $V/F$ be as before and denote as
before by  $\tilde{K}_v$ the compact open subgroup $\Spin_{\Lambda_v}(\fo_v)=\{ \tau \in \Spin_V(F_v)\mid 
\tau(\Lambda_v)=\Lambda_v\}$
for each finite place $v$ of $F$. 

Let $w$ be a fixed finite place of
$F$ and $T_w$ a regular isotropic 
subspace of $V_w=V\otimes F_w$ with $\dim(T_w)\ge 3$. Let
$G_w=\Spin_V(F_w)$, $H_w=\Spin_{T_w}(F_w)$ and 
$$\Gamma :=\Spin_V(F)\cap \Spin_V(F_w)\prod_{v\ne w}\tilde{K}_v.$$
Let a sequence $(W_i)_{i\in \N}$ of subspaces $W_i$ of $V$ (over the
global field $F$) be given such that $(W_i)^\perp_w=\xi_iT_w$ for each $i$
with elements $\xi_i$ from a fixed compact subset of $G_w$.

\vspace{3mm}
Then one has: If no infinite subsequence of the $W_i$ has a nonzero
intersection, the sets $\Gamma\backslash\Gamma\xi_iH_w$ are becoming
dense in $\Gamma\backslash G_w$ as $i\rightarrow \infty$, i.\ e., for every
open subset $U$ of $G_w$ one has $U\cap \Gamma\xi_iH_w\ne \emptyset$ for
sufficiently large $i$.
\end{proposition}
\begin{proof}
This is stated as {\it Claim} in 2.5 (p. 269) of \cite{ev}, it is
proven there using results of Ratner and Margulis/Tomanov from ergodic
theory.
\end{proof}
\begin{proposition}\label{intersections} 
Let $V/F, Q$ be as before and $\Lambda$ a lattice of
  full rank $n$ on $V$,
let $w$ be a fixed finite place of
$F$. 

Let a sequence $(W_i)_{i\in \N}$ of regular subspaces $W_i$ of
$V$ (over the global field $F$) of dimension $m\le \dim(V)-3$  be
given with the property that no 
infinite subsequence of the $W_i$ has a nonzero intersection and such
that the orthogonal complement of $(W_i)_w$ in $V_w$ is isotropic.

Assume moreover that the $w$-adic additive
valuations $\ord_w(\disc((W_i)_w\cap \Lambda_w))$ remain bounded by
some fixed integer $j$ as $i$ varies.

\vspace{2mm}
Then $\tilde{M}_i=W_i\cap \Lambda$ is represented primitively by all lattices
in the spinor genus 
$\spn(\Lambda)$ for sufficiently large $i$. 
\end{proposition}
\begin{proof}
This is stated and proven in \cite{ev} as Proposition 2 with
$j=1$. The proof doesn't change for arbitrary $j$; for the reader's
convenience we go through it and reformulate it slightly. 

As before we put $\tilde{K}_v=\Spin_{\Lambda_v}(\fo_v)$ for all finite places $v$ of $F$
and $$\Gamma :=\Spin_V(F)\cap \Spin_V(F_w)\prod_{v\ne w}\tilde{K}_w.$$

By Lemma \ref{orbits} the $(\tilde{M}_i)_w$ (and hence the $(W_i)_w$) fall into
finitely many orbits under the action of the compact open group
$\tilde{K}_w$, by treating the orbits separately we may therefore
assume that they all belong to the same orbit, i.e., with
$T_w=(W_1)^\perp_w$ we have  $(W_i)^\perp_w=\xi_iT_w$ with $\xi_i \in
\tilde{K}_w$ for all $i$. 

Given an arbitrary isometry class in the spinor genus of $\Lambda$  we
use Lemma \ref{p-rep} to obtain a representative 
$\tilde{\Lambda}\subseteq V$ of that isometry class with
$\tilde{\Lambda}_v=\Lambda_v$ for all finite places $v \ne w$
of $F$ and 
$\tilde{\Lambda}_w=g_w \Lambda_w$ for some $g_w\in G_w$.

By Proposition \ref{claim} for every open set $U\subseteq G_w$ there
is an $i_0$ such that we have $U\cap \Gamma\xi_iH_w\ne \emptyset$
for $i\ge i_0$. 

Taking $$U=g_w\tilde{K}_w\subseteq G_w$$ we obtain
$i_0$ such that for all $i \ge i_0$ one has elements $\gamma_i \in
\Gamma, \eta_i \in H_w, \kappa_i \in \tilde{K}_w$ with
$g_w\kappa_i=\gamma_i\xi_i\eta_i$. The lattice
$\Lambda_i':=\gamma_i^{-1}\tilde{\Lambda}$
is then in the isometry class of $\tilde{\Lambda}$; it satisfies
$(\Lambda_i')_v=\Lambda_v$ for all finite $v \ne w$ and
$$(\Lambda_i')_w=\gamma^{-1}_ig_w\Lambda_w=\xi_i\eta_i\Lambda_w=\xi_i\eta_i\xi_i^{-1}\Lambda_w.$$ 

From this and $\xi_i\eta_i\xi_i^{-1}\vert_{(W_i)_w}={\rm
  Id}_{(W_i)_w}$ we see $\tilde{M}_i=W_i\cap \Lambda'_i$, i.e., we have the
requested primitive representation by a lattice in 
the given isometry class. 
\end{proof}
\begin{proposition}\label{sequence}
Let $(V,Q),\Lambda$ be as before, let $(M_i)_{i \in \N}$ be a sequence
of $\fo$-lattices of rank $m\le n-3$ and assume that with some fixed
$c\in\fo, j\in \N$ and some fixed finite place $w$ of $F$ one has
\begin{enumerate}
\item $M_i$ is represented locally everywhere by $\Lambda$ with
  imprimitivity bounded by $c$ and with isotropic orthogonal
  complement at the place $w$ for all $i$.
\item $\ord_w(\disc((M_i)_w))\le j$ for all $i$.
\item The sequence $(\mu(M_i))_{i\in \N}$ of the minima of the $M_i$ tends to
  infinity.
\end{enumerate}
Then there is an $i_0\in \N$ such that $M_i$ is represented (with
imprimitivity bounded by $c$) by all isometry classes in the genus of
$\Lambda$ for all $i\ge i_0$.
\end{proposition}
\begin{proof}
By Lemma \ref{spinrep} we may concentrate on lattices in the spinor
genus of $\Lambda$ and assume that all $M_i$ are represented by
$\Lambda$ with
imprimitivity bounded 
by $c$; we may also replace the $M_i$ by their images under these
representations and assume $M_i\subseteq \Lambda$.  
We denote by $W_i$ the $F$-space generated by $M_i$ and assume that
there is an infinite subsequence of the $W_i$ with nonzero
intersection.

If $\x\ne {\bf 0}$ is a vector in this intersection we may assume
(replacing it by a suitable multiple) $\x\in \Lambda$ and hence $\x
\in W_i\cap \Lambda$ for infinitely many $i$. But then the condition
of bounded imprimitivity of the representations of the $M_i$ implies
$c\x \in M_i$ for infinitely many $i$, i.e., we have a vector of
length $N^F_\Q(c^2Q(\x))$ in infinitely many of the $M_i$, which
contradicts the assumption iii) that the minima of the $M_i$ tend to
infinity. 

We have shown that no infinite subsequence of the $W_i$ has a nonzero
intersection, so by Proposition \ref{intersections} there is an $i_0$
such that $W_i\cap \Lambda$ is represented primitively by all lattices
in the spinor genus of $\Lambda$ for all $i \ge i_0$, which implies
the assertion about 
the $M_i$.
\end{proof}
\begin{theorem}
Let $(V,Q), \Lambda$ be as before, fix a finite place $w$ of $F$ and
$j \in \N, c\in \fo$.

Then there exists a constant $C:=C(\Lambda,j,w,c)$ such that $\Lambda$
represents all $\fo$ - lattices $M$ of rank $m\le n-3$ satisfying
\begin{enumerate}
\item $M$ is represented by $\Lambda$ locally everywhere with
  imprimitivity bounded by $c$ and with isotropic orthogonal
  complement at the place $w$.
\item $\ord_w(\disc(M_w))\le j$
\item The minimum of $M$ is $\ge C$.
\end{enumerate}
The representation may be taken to be of imprimitivity bounded by $c$.

The isotropy condition is satisfied automatically if $m\le n-5$ or if
$w$ is such that $\disc(\Lambda_w)$ and $\disc(M_w)$ are units at $w$.   
\end{theorem}
\begin{proof}
If the assertion were wrong we could take a sequence of lattices $M_i$
of rank $m$ satisfying i) and ii) above and in addition $\mu(M_i)\ge
i$ such that no $M_i$ is represented by all classes in the genus of
$\Lambda$.
This would contradict the previous proposition.

The final remark about the isotropy condition follows easily from
known facts about quadratic spaces over non-archimedean local fields. 
\end{proof}
{\it Remark.}
\begin{enumerate}
\item All arguments used above remain valid if one imposes additional
  congruence conditions at finite places outside the fixed place $w$
  on the representations, starting out from the versions with
  congruence conditions of Lemma \ref{spinrep} and Lemma \ref{p-rep}
  as indicated in 
  the remarks after these Lemmata and incorporating the congruence
  conditions into the choice of the compact open groups $\tilde{K}_v$
  for the $v \ne w$.
\item The Theorem without the condition on
  $\ord_w(\disc(M))$ and without the condition on bounded
  imprimitivity has been proven in \cite{hkk} by a different method
  for $n\ge 2m+3$. 

\parindent=0pt
We note that for  $n\ge 2m+3$ it is an easy consequence of the local theory of
  lattices that there is a $c \in \fo$ such that a lattice of rank $M$
  which is represented locally everywhere by $\Lambda$ is in fact
  represented locally everywhere by $\Lambda$  with imprimitivity
  bounded by $c$ (see e.\ g. \cite[Theorem 2]{kit_densities}), so that (without
  using \cite{hkk}) we can omit 
  the condition of bounded imprimitivity of the local representations
  if $n\ge 2m+3$ holds. More generally the same is true if the Witt
  index of $V$ is at least $m$.
\item From the point of view of the analytic theory the condition of
  bounded imprimitivity for the local representations is natural: If
  one splits the number of representations of $r(\Lambda, M)$ of $M$
  by $\Lambda$ into the sum of a main term given by Siegel's weighted
  average   $r(\gen(\Lambda),M)$ and an error term, a calculation of
  the local densities shows that  the main term grows only slowly (or
  not at all) if $M$ varies in a sequence of lattices represented
  locally only with growing imprimitivity, see for example \cite[Theorem
  5.6.5 c)]{Ki}. 

\parindent=0pt
The usual method of showing
  that the main term grows faster in absolute value than the error
  term so that the sum 
  must eventually become positive can therefore not work if one
  considers such a sequence, and it even appears plausible that the
  small average number of representations may be due to the fact that
  not all isometry classes do represent the lattices considered.
\item 
 For applications it would sometimes be desirable to have a
  different condition than growing minimum of the lattices to be
  represented, e.\ g.\ growing discriminant plus representability of
  all successive minima. It appears that at least the present method
  is not able to give such a result: If we consider an infinite
  sequence of lattices $M_i \subseteq \Lambda$ whose minimum is
  bounded, there must be infinite subsequences having a nonzero intersection,
  since there are only finitely many vectors of given length in
  $\Lambda$. Propositions \ref{claim} and \ref{intersections} can
  therefore not be applied to such a situation. One can however easily
  adapt the method to the case of extensions of representations of
  some fixed lattice as
  treated in \cite{boe-ra,ckko}, see Corollary \ref{extensions} below. 
\item At least presently there seems to be no way to obtain an
  effective bound on the constant $C$ of the Theorem. The result is
  therefore probably more useful in the form of Proposition
  \ref{sequence}. We note that for $n\ge 2m+3$ the method of
  \cite{hkk} does at least in principle allow to find an effective bound.
\end{enumerate}  
\begin{corollary}\label{extensions}
Let $(V,Q), \Lambda$ be as before, fix a finite place $w$ of $F$ and
$j \in \N, c\in \fo$.

Let $R\subseteq \Lambda$ be a fixed $\fo$-lattice of rank $r$,
$\sigma:R\longrightarrow \Lambda$ a representation of $R$ by
$\Lambda$ and assume that $R_w$ is unimodular.

Then there exists a constant $C:=C(\Lambda,R,j,w,c)$ such that one
has:
If $M\supseteq R$ is an  $\fo$-lattice of rank $m\le n-3$ and
\begin{enumerate}
\item For each place $v$ of $F$ there is a representation
  $\tau_v:M_v\longrightarrow \Lambda_v$ with $\tau_v\vert_{R_v}=\sigma_v$
  with imprimitivity bounded by $c$ and with isotropic orthogonal
  complement in $\Lambda$ at the place $w$
\item $\ord_w(\disc(M_w))\le j$
\item The minimum of $M\cap (FR)^\perp$ is $\ge C$,
\end{enumerate}
then there exists a representation $\tau:M\longrightarrow \Lambda$
with $\tau \vert_{R}=\sigma$.

The representation may be taken to be of imprimitivity bounded by $c$.

The isotropy condition is satisfied automatically if $m\le n-5$ or if
$w$ is such that $\disc(\Lambda_w)$ and $\disc(M_w)$ are units at $w$.   
\end{corollary}
\begin{proof} This is proven in the same way as Theorem IV' in
  \cite[p.95]{boe-ra} and Theorem 2.1 in \cite{ckko}, namely by constructing
  a representation $\rho$ of $M\cap (FR)^\perp$ into $\Lambda \cap
  (FR)^\perp$ satisfying suitable congruence conditions at the primes
  dividing the discriminant of one of $\Lambda, R$ and pasting it
  together with the given $\sigma$, the congruence conditions being
  chosen so that $\rho \perp \sigma$ actually maps $M$ into $\Lambda$.
Notice for this that the condition of bounded imprimitivity implies
that the index of  $M\cap (FR)^\perp$ in the orthogonal projection
$\pi(M)$ of $M$ onto $(FR)^\perp$ is bounded independently of $M$ (for
an integral primitive sublattice this index is bounded by the index of $R$ in
its dual lattice $R^\#$). Notice (for comparison with Theorem 2.1 of
\cite{ckko})  that the boundedness of this
index also implies that the quotient of the minimum of  $M\cap
(FR)^\perp$ and the minimum of $\pi(M)$ is bounded by some constant
independent of $M$ so that is relevant only for the size of the
constant  $C$ which of these minima one considers.
\end{proof} 
 Results of Kitaoka allow to lift the condition
  on bounded imprimitivity in some cases with $n<2m+3$ too:
\begin{corollary}\label{kitaoka_cor}
Let $F=\Q$, let $(V,Q), \Lambda$ be as before and  fix a prime $q$
 and 
$j \in \N$.
\begin{enumerate}
\item Let $m \ge 6$ and $n=\dim(V)\ge 2m$.
Then there exists a constant $C:=C(\Lambda,j,q)$ such that $\Lambda$
represents all $\Z$ - lattices $M$ of rank $m$ which are represented
by $\Lambda$ locally everywhere, have minimum $\ge C$  and satisfy
$\ord_q(\disc(M))\le j$. 
\item  Let $m \ge 3$ and $n=\dim(V)\ge 2m+1$. Then there exists a
  constant $C:=C(\Lambda,j,q)$ such that $\Lambda$ 
represents all $\Z$ - lattices $M$ of rank $m$ which are represented
by $\Lambda$ locally everywhere, have minimum $\ge C$, satisfy
${\ord_q(\disc(M))\le j}$ and which are in the case $m=3$ such that the
orthogonal complement of $M_q$ in $\Lambda_q$ is isotropic.
\item  Let $m=2$ and $n=\dim(V)\ge 6$. Then there exists a
  constant $C:=C(\Lambda,j,q)$ such that $\Lambda$ 
represents all $\Z$ - lattices $M$ of rank $m$ which are represented
by $\Lambda$ locally everywhere, have minimum $\ge C$, satisfy
$\ord_q(\disc(M))\le j$ and which are such that the
orthogonal complement of $M_q$ in $\Lambda_q$ is isotropic.
\item Let a positive definite $\Z$-lattice $M_0$ of rank $m\le n-3$ with Gram
  matrix $T_0$ be given. Let $S$ be a finite set of primes with $q\in
  S$ such that one has
\begin{enumerate}
\item $\Lambda_p$ and $M_p$ are unimodular  for all primes $p\not\in
  S$ and for $p=q$.
\item Each isometry class in the genus of $\Lambda$ has a
  representative $\Lambda'$ on $V$ such that $\Lambda_p'=\Lambda_p$
  for all primes $p\not\in S$.
\end{enumerate}
Then there exists a constant $C:=C(\Lambda,T_0,S)$ such that for all
sufficiently large integers $a\in \Z$ which are not divisible by a prime in $S$,
the $\Z$-lattice $M$ with Gram matrix $aT_0$ is represented by
$\Lambda$ if it is represented by all completions $\Lambda_p$.
\end{enumerate}
\end{corollary}
\begin{proof}
We notice first that the conditions in the various cases imply that
the orthogonal complement of $M_q$ in $\Lambda_q$ is isotropic.

In the articles \cite{kit_densities,kit_nagoya115,kit_nagoya133,kit_nagoya141}
Kitaoka proved that the following condition ${\rm \bf R_{m,n}}$  is satisfied
for the values $m,n$ given in the assertions (with $M$ restricted to
scalings of a fixed $M_0$ as described above in case iv)):
\begin{quote}
{\it ${\rm \bf (R_{m,n})}$: For any $C_1>0$ there exists a $C_2>0$ such that the
  following is true: If $M$ is represented locally everywhere by
  $\Lambda$ and satisfies $\mu(M)\ge C_2$, there exists a lattice
  $M'\supseteq M$ of rank $m$ with minimum $\mu(M')\ge C_1$ which is
  represented by $\Lambda$ locally everywhere primitively.}
\end{quote}
Since for $C_1$ large enough the Theorem asserts that $M'$ is
represented globally primitively by $\Lambda$, we obtain the
representation of $M$ by $\Lambda$ if $C$ in the Corollary is $\ge C_2$. 
\end{proof}

For the convenience of more matrix oriented readers we conclude by
giving a matrix version (not in ``geometrischer Einkleidung'', see
\cite{siegel}) of the main result in the case $F=\Q$:

\begin{theorem}\label{matrixversion}
Let $S\in M_n^{\rm sym}(\Z)$ be a positive definite integral  symmetric
$n\times n$-matrix, fix a prime $q$ and positive integers $j,c$.

Then there is a constant $C$ such that a positive definite matrix $T \in
M_m^{\rm sym}(\Z)$ with $m\le n-3$ is represented by $S$ (i.e., 
$T={}^t\!XSX$ with $X \in M_{nm}^{\rm sym}(\Z)$) provided it satisfies:
\begin{enumerate}
\item For each prime $p$ there exists a matrix $X_p \in M_{nm}(\Z_p)$ with
  ${}^t\!X_pSX_p=T$ such that 
  the  elementary divisors of $X$ divide $c$ and such that the equations
${}^t\!X_qS\y={\bf 0}$ and ${}^t\!\y S\y=0$ have a nontrivial common
solution $\y \in \Z_q^n$  
\item $q^j \nmid \det(T)$
\item $\min\{{}^t\!\y T\y \mid {\bf 0}\ne \y\in  \Z^m\}>C$
\end{enumerate} 
The matrix $X$ may be chosen to have elementary divisors dividing $c$.
\end{theorem}

\enlargethispage{8mm}
\vspace{0.1cm}
Author's address:
Rainer Schulze-Pillot,
Fachrichtung 6.1 Mathematik,\\
Universit\"at des Saarlandes (Geb. E2.4),
Postfach 151150, 66041 Saarbr\"ucken, Germany\\
email: schulzep@math.uni-sb.de


\begin{thebibliography}{MVW}
\bibitem{boe-ra} 
S. B\"ocherer, S. Raghavan:
On Fourier coefficients of Siegel modular forms.
J. Reine Angew. Math. {\bf 384} (1988), 80--101. 
\bibitem{ckko} W. K. Chan, B. M. Kim, M.-H. Kim, B.-K. Oh: Extensions
  of representations of integral quadratic forms, Ramanujan J. 2008
  (Electronic version at: DOI 10.1007/s11139-007-9023-y)
\bibitem{eichler_qfog} M. Eichler: Quadratische Formen und orthogonale
  Gruppen.  {\it Springer-Verlag Berlin-New York}, 2. ed. 1974.
\bibitem{eichler_spin} M. Eichler: Die \"Ahnlichkeitsklassen
  indefiniter Gitter. Math. Z.  {\bf 55}  (1952), 216--252.  
\bibitem{ev} J. Ellenberg, A. Venkatesh: Local-global principles for
  representations of quadratic forms.  Invent. Math.  {\bf 171}  (2008),
  257--279.
\bibitem{hsia-pacific} J. S. Hsia: Representations by spinor genera.
  Pacific J. Math.  {\bf 63} (1976), 147--152. 
\bibitem{hkk}  J. S. Hsia, Y. Kitaoka,  M. Kneser: Representations of
  positive definite quadratic forms. 
J. Reine Angew. Math. {\bf 301} (1978), 132--141. 
\bibitem{kit_densities} Y. Kitaoka: Local densities of quadratic
    forms.  Investigations in number theory,  433--460,
    Adv. Stud. Pure Math., 13, Academic Press, Boston, MA, 1988. 
\bibitem{kit_nagoya115} Y. Kitaoka: Some remarks on representations
    of positive definite quadratic forms.  Nagoya Math. J.  115
    (1989), 23--41. 
\bibitem{kit_nagoya133} Y. Kitaoka: The minimum and the primitive
    representation of positive definite quadratic forms.  Nagoya
    Math. J.  133  (1994), 127--153. 
\bibitem{kit_nagoya141} Y. Kitaoka: The minimum and the primitive
  representation of positive definite quadratic forms. II.  Nagoya
  Math. J.  141  (1996), 1--27. 
\bibitem{Ki}  Y. Kitaoka: Arithmetic of quadratic forms. 
{\it Cambridge Tracts in Mathematics} {\bf 106}.
Cambridge University Press, Cambridge, 1993. 
\bibitem{kneser_mz} M. Kneser: Darstellungsma\ss e indefiniter
  quadratischer Formen. Math. Z.  {\bf 77}  (1961), 188--194.
\bibitem{kneser_buch} M. Kneser: Quadratische Formen. Springer-Verlag 2002
\bibitem{OM} O.T. O'Meara: Introduction to quadratic forms,
Springer-Verlag 1973. 
\bibitem{siegel} C.\ L.\ Siegel: Vorwort ``Zur Reduktionstheorie
  quadratischer Formen'', p.329-330 in Gesammelte Abhandlungen IV,
  Springer-Verlag Berlin-Heidelberg-New York 1979.

\end{thebibliography}
\end{document}